\documentclass[12pt,a4paper]{amsart}
\usepackage[left=3.5cm,right=3.5cm,top=3.5cm,bottom=3.5cm]{geometry}
\usepackage{enumerate}
\usepackage{leftidx}
\usepackage{mathtools,setspace}
\usepackage{amsmath,amscd,amssymb,amsthm,mathrsfs,stmaryrd}
\usepackage[arrow,matrix]{xy}
\usepackage{graphicx,color,tikz}
\usepackage[colorlinks,linkcolor=blue,anchorcolor=blue,citecolor=blue,backref=page]{hyperref}
\usetikzlibrary {graphs}

\usetikzlibrary{positioning,bending}
\usetikzlibrary{arrows.meta}

\DeclareMathOperator{\gal}{Gal}

\DeclareMathOperator{\aut}{Aut}

\theoremstyle{definition}
\newtheorem{definition}{Definition}[section]
\newtheorem{example}[definition]{Example}
\newtheorem{remark}[definition]{Remark}

\theoremstyle{plain}
\newtheorem{theorem}[definition]{Theorem}
\newtheorem{corollary}[definition]{Corollary}
\newtheorem{lemma}[definition]{Lemma}
\newtheorem{proposition}[definition]{Proposition}

\newtheorem*{theorem*}{Theorem}

\newcommand{\F}{{ \mathbb F }}
\newcommand{\N}{{ \mathbb N }}
\newcommand{\Q}{{ \mathbb Q }}
\newcommand{\Z}{{ \mathbb Z }}

\newcommand{\p}{{ \mathfrak p }}

\newcommand{\m}{{ \mathfrak m }}

\renewcommand{\O}{{ \mathcal O }}

\author[A. Ferraguti]{Andrea Ferraguti}
\address{Universit\'a di Torino, Dipartimento di Matematica, via Carlo Alberto 10, 10123 Torino, Italy}
\email{and.ferraguti@gmail.com}

\title[]{Uncountably many settled elements are never Frobenius elements}

\begin{document}

\begin{abstract}
    Boston and Jones conjectured that given a number field $K$, a quadratic polynomial $f\in K[x]$, and an element $\alpha\in K$ such that the arboreal Galois representation $\rho_{f,\alpha}\colon \gal(\overline{K}/K)\to \aut(T_\infty)$ attached to the pair $(f,\alpha)$ is well defined, any element in the image of such representation that is a Frobenius for an unramified prime of $K$ is a settled element. We exhibit an explicit uncountable subset $S\subseteq \aut(T_\infty)$, consisting of pairwise non-conjugate settled elements, such that no matter how $K,f,\alpha$ are chosen, no element of $S$ can be a Frobenius for any ramified or unramified prime of $K$.
\end{abstract}

\maketitle
\section{Introduction}

Let $K$ be a number field, let $f\in K[x]$ be of degree $d\ge 2$ and let $\alpha\in K$. Under the condition that $\alpha$ is not in the image of the ramification locus of any iterate of $f$, one can construct an arboreal Galois representation
$$\rho_{f,\alpha}\colon \gal(\overline{K}/K)\to \Omega_\infty^d,$$
where $\Omega_\infty^d$ is the automorphism group of the infinite, regular, rooted, $d$-ary tree $T_\infty^d$. We refer to the survey \cite{jones} for more detail. The image of such representation is denoted by $G_\infty(f,\alpha)$, and is also referred to as the \emph{dynamical Galois group} of the pair $(f,\alpha)$ (and in fact it can be defined in the broader setting where $\alpha$ is allowed to lie in the image of some ramification locus).

Understanding which elements of $G_\infty(f,\alpha)$ can be Frobenius elements of some prime of $K$ is an intriguing and hard-looking task; of course as soon as $G_\infty(f,\alpha)$ contains uncountably many conjugacy classes of elements (which is the case for example when $[\Omega_\infty^d:G_\infty(f,\alpha)]<\infty$), then there are uncountably many conjugacy classes that cannot contain any Frobenius element, since there are only countably many conjugacy classes of Frobenius elements. It is however unknown how to characterize conjugacy classes that contain Frobenius elements.

In \cite{bosjon1}, Boston and Jones introduced the notion of \emph{settledness}. Given a polynomial $f$ of degree $d\ge 2$ over a field $F$, a polynomial $g\in F[x]$ is called \emph{$f$-stable} if $g(f^k)$ is irreducible for every $k\ge 1$. Fix $\alpha\in F$ and for every $n\ge 1$ let $e_n$ be the sum of the degrees of the $f$-stable factors of $f^n-\alpha$. We say that $(f,\alpha)$ is \emph{settled} if $\lim_{n\to \infty}e_n/d^n=1$. The notion of settledness carries over to elements of $\Omega_\infty^d$, as explained in \cite{bosjon2}. If $V_n$ is the set of nodes at distance $n$ from the root in $T_\infty^d$ and $\sigma\in \Omega_\infty^d$ we say that $v\in V_n$ \emph{lies in a stable cycle} with respect to $\sigma$ if, once we let $e$ be the length of the orbit of $v$, we have that for every $k>n$ any descendant of $v$ in $V_k$ lies in a $\sigma$-orbit of length $ed^{k-n}$. The element $\sigma$ is then called \emph{settled} if the proportion of nodes in $V_n$ that lie in a stable cycle with respect to $\sigma$ tends to $1$ as $n\to \infty$. The two notions are connected in the following way: suppose that $K$ is a number field, $f\in K[x]$ has degree $d\ge 2$ and $\alpha\in K$, and assume that $\rho_{f,\alpha}$ is well-defined. Let $\p$ be a prime of $K$ that is unramified in the direct limit $K_\infty(f,\alpha)$ of the splitting fields of $f^n-\alpha$ and such that $f$ and $\alpha$ are integral at $\p$. Let $\sigma_\p\in G_\infty(f,\alpha)\subseteq\Omega_\infty^d$ be a Frobenius element at $\p$. Since the factorization pattern of $f^n-\alpha$ modulo $\p$ coincides with the cycle decomposition of $\sigma_\p$ as a permutation of the roots of $f^n-\alpha$, we get that $\sigma_\p$ is settled if and only if $(f,\alpha)$ is settled modulo $\p$.

In \cite{bosjon1} it was conjectured that if $F$ is a finite field of odd characteristic and $f\in F[x]$ is a quadratic polynomial different from $x^2$, then $(f,0)$ is settled. Notice that if $f=(x-a)^2+b\in F[x]$, then $(f,0)$ is settled if and only if $({}^gf,-a)$ is settled, where ${}^{g}f(x)=f(x+a)-a$. Hence, the aforementioned conjecture is equivalent to asking that if $f=x^2+c\in \F_q[x]$ for some odd $q$, then $(f,\alpha)$ is settled unless $(c,\alpha)=(0,0)$. Boston and Jones' conjecture implies that if $K$ is a number field, $f=x^2+c\in K[x]$, $\alpha\in K$ and $\p$ is a prime of $K$ that is unramified in $K_\infty(f,\alpha)$ and such that $f$ and $\alpha$ are $\p$-integral, then any Frobenius at $\p$ is settled. Notice that when $f$ is post-critically finite, the set $S$ of primes that do not satisfy the above conditions is finite and hence, by Chebotarev density theorem, elements of $G_\infty(f,\alpha)$ that are a Frobenius for a prime outside $S$ are dense. We say that $G_\infty(f,\alpha)$ is \emph{densely settled} if it contains a dense subset of settled elements, and it follows from Boston and Jones' conjecture that $G_\infty(f,\alpha)$ is densely settled when $f$ is post-critically finite (see \cite{lukina} for recent developments on this problem). We remark the noteworthy work of Ejder \cite{ejder} that shows that dynamical Galois groups of PCF quadratic rational functions are not always densely settled.

Boston and Jones' conjecture implies that for quadratic polynomials over number fields, Frobenius elements of unramified primes in $G_\infty(f,\alpha)$ are settled. This raises a very natural question: which settled elements of $G_\infty(f,\alpha)$ are \emph{not} allowed to be Frobenius elements? Our main theorem is the following (see Theorem \ref{non_frobenius} and Remark \ref{explicit}).

\begin{theorem*}
    There exists an explicit uncountable subset $S\subseteq \Omega_\infty^2$ of pairwise non-conjugate settled elements with the following property: if $K$ is a number field, $f=x^2+c\in K[x]$, and $\alpha\in K$ is not postcritical, the set $\{\sigma\in G_\infty(f,\alpha)\colon \sigma \mbox{ is a Frobenius for some prime of }K\}$ does not intersect $S$.
\end{theorem*}

Notice that the set $S$ does not depend on $K,f,\alpha$: elements of $S$ cannot be a Frobenius no matter how the starting data is chosen. Moreover, there is no hypothesis on ramification: elements of $S$ are not even allowed to be Frobenius elements for ramified primes. Here by "Frobenius for a ramified prime $\p$" we mean a lift of a topological generator of $D_\p/I_\p$, where $D_p$ and $I_p$ are the decomposition and the inertia group at $\p$, respectively.

The proof consists of two steps. First, we construct an explicit injective map $\Psi\colon\{0,1\}^{\N_{\ge 1}}\to \Omega_\infty^2$ whose image consists of pairwise non-conjugate settled elements (see Proposition \ref{settled_properties}). In fact, not only these elements are non-conjugate but the natural map
$$(\phi_n)_{n\ge 1}\colon\Omega_\infty^2\to (\Omega_\infty^2)^{\text{ab}}\cong \F_2^{\N_{\ge 1}}$$
restricts to a bijection on the image of $\Psi$. Next, we prove that elements $\sigma\in \Omega_\infty^2$ that lie in the image of an arboreal representation of a quadratic polynomial $f$ over a $p$-adic field with a non-fixed basepoint $\alpha$ have the property that for some integers $a\ge 1$ and $b\ge 0$, the sequence $(\phi_{ak+b}(\sigma))_{k\ge 1}$ is preperiodic (see Corollary \ref{cor_per}). Now it is easy to conclude, since of course $\F_2^{\N_{\ge 1}}$ contains uncountably many elements $(t_n)_{n\ge 1}$ such that $(t_{ak+b})_{k\ge 1}$ is not preperiodic for any $a,b$. Moreover, one can use the descriptions of $(\phi_n)_{n\ge 1}$ and $\Psi$ to exhibit an explicit uncountable subset $T\subseteq \{0,1\}^{\N_{\ge 1}}$ with the property that no element in $\Psi(T)$ is allowed to be a Frobenius.

\section*{Human-AI collaboration}
This paper was ideated and written between May 2024 and June 2025. No LLM has been used to write or contribute to any part of the paper. 

\section*{Acknowledgments}
The idea for this paper originated during the workshop "Dynamics on zero-dimensional spaces: new connections" which was held at the Lorentz Center in Leiden, in May 2024. I am grateful to everyone who took part to the organization of the workshop for the great working conditions that were provided. I am also deeply indebted to Carlo Pagano for providing the key insight of Lemma \ref{preperiodic_orbit}, and for highly instructive conversations, as always. I have been partially supported by the
“National Group for Algebraic and Geometric Structures, and their Applications" (GNSAGA
- INdAM).

\section{Constructing settled elements}\label{constructing_settled}
From now on, we let $\Omega_\infty=\Omega_\infty^2$ be the automorphism group of the infinite, rooted, binary tree $T_\infty$, and we denote by $V_n$ the set of nodes of the tree at distance $n$ from the root. We recall that every element $\sigma\in \Omega_\infty$ is uniquely described by its portrait. Namely, we can write
$$\sigma=(\varepsilon_1^1(\sigma),\varepsilon_2^1(\sigma)\varepsilon_2^2(\sigma),\ldots,\varepsilon_n^1(\sigma)\ldots \varepsilon_n^{2^{n-1}}(\sigma),\ldots)$$
where $\varepsilon_i^j(\sigma)\in \F_2$ for every $i,j$. Here we are adopting the convention of labeling the nodes at level $n$ in $T_\infty$ from left to right with $v_1^n,\ldots,v_{2^n}^n$; with this convention, suppose that $v_i^n\in V_n$, $v_{\lceil i/2\rceil}^{n-1}$ is the ancestor of $v_i$ at level $n-1$, $\sigma\in \Omega_\infty$ and $v_j^{n-1}=\sigma(v_{\lceil i/2 \rceil}^{n-1})$. If $v_k^n,v_{k+1}^n$ are the descendants of $v_j^{n-1}$ at level $n$, where $k\coloneqq 2(j-1)+1$, then 
\begin{equation}\label{sign_parity}
 \sigma(v_i^n)=v_k^n \mbox{ if and only if } i+k\equiv \varepsilon_n^j(\sigma) \bmod 2.
\end{equation}
In practice, to ease the notation we will write $v_i$ instead of $v_i^n$, as long as there is no ambiguity about the level of $v_i$.

    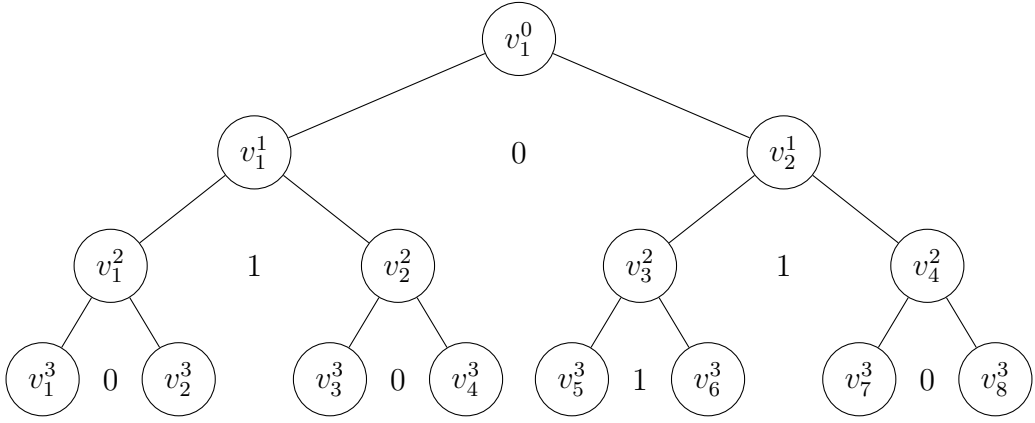
\begin{figure}[h]
\centering

\begin{tikzpicture}

\node[circle,draw]{$v_1^0$}[sibling distance=70mm]
child{node[circle,draw](a){$v_1^1$}[sibling distance=38mm]
child{node[circle,draw](c){$v_1^2$}[sibling distance=18mm]
child{node[circle,draw](g){$v_1^3$}}
child{node[circle,draw](h){$v_2^3$}}}
child{node[circle,draw](d){$v_2^2$}[sibling distance=18mm]
child{node[circle,draw](i){$v_3^3$}}
child{node[circle,draw](j){$v_4^3$}}}
}
child{node[circle,draw](b){$v_2^1$}[sibling distance=38mm]
child{node[circle,draw](e){$v_3^2$}[sibling distance=18mm]
child{node[circle,draw](k){$v_5^3$}}
child{node[circle,draw](l){$v_6^3$}}}
child{node[circle,draw](f){$v_4^2$}[sibling distance=18mm]
child{node[circle,draw](m){$v_7^3$}}
child{node[circle,draw](n){$v_8^3$}}}
};
\path (a) -- (b) node [midway] {0};
\path (c) -- (d) node [midway] {1};
\path (e) -- (f) node [midway] {1};
\path (g) -- (h) node [midway] {0};
\path (i) -- (j) node [midway] {0};
\path (k) -- (l) node [midway] {1};
\path (m) -- (n) node [midway] {0};

\end{tikzpicture}\caption{An element with portrait $(0,11,0010,\ldots)$ acts on $V_3$ as the permutation $(13)(24)(5768)$}\label{fig:tree}
\end{figure}

\begin{lemma}\label{orbits}
	Let $\sigma\in \Omega_\infty$ and let $n\ge 0$. Let $\mathcal O\coloneqq\{v_{i_1},\ldots,v_{i_m}\}$ be an orbit of $\sigma$ acting on $V_n$, where $i_j\in \{1,\ldots,2^n\}$ for every $j$, the $v_{i_j}$'s are all distinct nodes in $V_n$, $\sigma(v_{i_j})=v_{i_{j+1}}$ for every $j=1,\ldots,m-1$ and $\sigma(v_{i_{m}})=v_{i_1}$.
	\begin{enumerate}[(a)]
		\item Suppose that there exists a unique $k\in \{1,\ldots,m\}$ such that 
		$$\varepsilon_{n+1}^{i_j}(\sigma)=\begin{cases}1 & \mbox{if } j=k\\0 & \mbox{otherwise}\end{cases}.$$ Then the descendants of the $v_{i_j}$'s in $V_{n+1}$ form a unique $\sigma$-orbit.
		\item Suppose that for every $j\in \{1,\ldots,m\}$ we have $\varepsilon_{n+1}^{i_j}(\sigma)=0$. Then the descendants of the $v_{i_j}$'s in $V_{n+1}$ partition in two disjoint $\sigma$-orbits each of length $m$.
		\end{enumerate} 
\end{lemma}
\begin{proof}
	$(a)$ First of all, up to a cyclic relabeling of the $v_{i_j}$'s we can assume without loss of generality that $k=1$. We first claim that $\varepsilon_{n+1}^{i_j}(\sigma^m)=1$ for every $j\in \{1,\ldots,m\}$. Let $j\in \{1,\ldots, m\}$ and let $w_{\ell_1}$ be a descendant of $v_{i_j}$ in $V_{n+1}$, for some $\ell_1\in \{2(i_j-1)+1,2(i_j-1)+2\}$. Since for $j\ne 1$ we have that $\varepsilon_{n+1}^{i_j}(\sigma)=0$ it follows, by applying \eqref{sign_parity} recursively $m-j$ times, that 
	\begin{equation}\label{first_piece}
		\sigma^{m-j}(w_{\ell_1})=w_{\ell_2} \mbox{ with }\ell_1 \equiv \ell_2\bmod 2.
	\end{equation}
	Now $w_{\ell_2}$ is a descendant of $v_{i_m}$ by construction, and since $\varepsilon_{n+1}^{i_1}(\sigma)=1$ by hypothesis we have by \eqref{sign_parity} that $\sigma(w_{\ell_2})=w_{\ell_3}$ with $\ell_2\equiv \ell_3+1\bmod 2$, and consequently by \eqref{first_piece} we get:
	\begin{equation}\label{second_piece}
		\sigma^{m-j+1}(w_{\ell_1})=w_{\ell_3} \mbox{ with } \ell_1\equiv \ell_3+1\bmod 2.
	\end{equation}
	Finally, applying \eqref{sign_parity} to \eqref{second_piece} $j-1$ times we get that
	$$\sigma^m(w_{\ell_1})=\sigma^{m-j+1+(j-1)}(w_{\ell_1})=w_{\ell_1'} \mbox{ with }\ell_1\equiv \ell_1'+1\bmod 2.$$
	On the other hand $\sigma^{m}$ is the identity when acting on $(v_{i_1},\ldots,v_{i_{m}})$, and hence $w_{\ell_1'}$ must be the brother of $w_{\ell_1}$. But this means exactly that $\varepsilon_{n+1}^{i_j}(\sigma^m)=1$.

	To conclude the proof, we have to show that if $w,w' \in V_{n+1}$ are descendants of some nodes in $\mathcal O$, then $\sigma^N(w)=w'$ for some $N$. Since $\sigma$ acts transitively on $(v_{i_1},\ldots,v_{i_m})$, there exists $M$ such that $\sigma^M(w)$ and $w'$ have the same parent. If $\sigma^M(w)=w'$, we are done. Otherwise, $\sigma^m(\sigma^M(w))=w'$, since by what we proved above $\sigma^m$ acts on $\sigma^M(w)$ by swapping it with its brother.
	
	$(b)$ Let $w_\ell$ be a descendant of $v_{i_1}$. Since $\sigma$ acts transitively on $\mathcal O$, it is obvious that the $\sigma$-orbit of $w_\ell$ has at least $m$ elements. On the other hand it using \eqref{sign_parity} $m$ times it follows immediately that if $\sigma^k(w_\ell)=w_{h}$ for some $h\in \{1,\ldots,2^{n+1}\}$ and some $k\in \{1,\ldots,m\}$, then necessarily $\ell\equiv h\bmod 2$. This shows that the $\sigma$-orbit of $w_\ell$ has exactly $m$ elements. The same argument applies to the brother $w_{\ell'}$ of $w_\ell$, and since $\ell\not\equiv \ell'\bmod 2$, the two orbits must be disjoint.
\end{proof}

From now on, we let $\N_1\coloneqq \N\setminus\{0\}$.

We are now going to construct an injective map
$$\Psi\colon\{0,1\}^{\N_1}\to \Omega_\infty$$
by describing the portrait of $\Psi(s)$ for every $s\in \{0,1\}^{\N_1}$. This is done recursively in the following way. Let $s=(s_1,s_2,\ldots,s_n,\ldots)\in \{0,1\}^{\N_1}$; for every $n\ge 1$, $\Psi(s)$ acts on $V_n$ as an element of $S_{2^n}$, and as such it decomposes into $d_n$ disjoint cycles $c_1,\ldots,c_{d_n}$, which we order by increasing length.\footnote{It does not really matter how we choose to break ties, but for the sake of choosing a way, if $c,c'$ are cycles of the same length then they are elements of $S_{2^n}$ via the labeling of $V_{n}$ we fixed at the beginning, and hence they contain nodes uniquely labeled with elements of $\{1,\ldots,2^n\}$. We write $c$ first if the least index of its nodes is smaller than that of $c'$.} For every $k\in \{1,\ldots,d_n\}$, we let $v_{i_k}\in V_n$ be the node with the smallest index among those that belong to the orbit of $c_k$, so that $i_k\in \{1,\ldots,2^n\}$ for every $k$.\footnote{Once again, this choice is not really relevant, we could pick any one node from every cycle, but for the sake of defining a unique function, we made this choice.} Now the $(n+1)$-th entry of the portrait of $\Psi(s)$ is defined in the following way: if $s_{n+1}=0$ then we let $\varepsilon_{n+1}^i(\Psi(s))=1$ if $i=i_k$ for some $k\in \{1,\ldots,d_n\}$, and $\varepsilon_{n+1}^i(\Psi(s))=0$ otherwise. If $s_{n+1}=1$, then we let $\varepsilon_{n+1}^i(\Psi(s))=1$ if $i=i_k$ for some $k\in \{2,\ldots,d_n\}$, and $\varepsilon_{n+1}^i(\Psi(s))=0$ otherwise.

\begin{example}
	Let $s\in \{0,1\}^{\N_1}$ be a sequence with $s_1=1$, $s_2=0$ and $s_3=1$. To describe the first three entries of the portrait of $\Psi(s)$, we start from $\Psi(s)$ acting on $V_0$. This is a set with one element, so $\Psi(s)$ decomposes at level $0$ as a single cycle $c_1$. Since $s_1=1$, the unique digit of the first entry of the portrait of $\Psi(s)$ is the one corresponding to $c_1$, and it must be a $0$. Hence $\varepsilon_1^1(\Psi(s))=0$, $\Psi(s)$ acts as the identity on $V_1$, and it decomposes in $S_2$ as two disjoint $1$-cycles $c_1c_2=(v_1)(v_2)$. Now we have to fix, for every cycle, the node with the least index. Since $s_2=0$, the construction tells us that $\varepsilon_2^1(\Psi(s))=\varepsilon_2^2(\Psi(s))=1$. It follows that $\Psi(s)$ acts on $V_2$ as a product of two $2$-cycles $c_1=(v_1v_2)$ and $c_2=(v_3v_4)$. Now we fix in $V_2$, for each of these cycles, the node with the least index, namely $v_1$ and $v_3$. Since $s_3=1$, we must have $\varepsilon_3^i(\Psi(s))=1$ if $i=3$ and $\varepsilon_3^i(\Psi(s))=0$ otherwise. Hence the portrait of $\Psi(s)$ will look like:
	$$(0,11,0010,\ldots).$$
    As shown in Figure \ref{fig:tree}, this acts on $V_3$ as the permutation $(13)(24)(5768)$.
\end{example}

Before describing the properties of $\Psi$, recall that for every $n\ge 1$ there is a continuous group homomorphism
$$\phi_n\colon \Omega_\infty\to \F_2$$
$$\sigma\mapsto \sum_{i=1}^{2^{n-1}}\varepsilon_n^i(\sigma).$$
More details can be found in \cite{ferpag3,ferpag2,ferpag1}.

\begin{proposition}\label{settled_properties}
    The following hold true.
    \begin{enumerate}[(a)]
        \item For every $s=(s_i)_{i\ge 1}\in \{0,1\}^{\N_1}$, the number of orbits of $\Psi(s)$ acting on $V_n$ is 
        $$1+\sum_{i=1}^ns_i.$$
        \item For every $s\ne t\in \{0,1\}^{\N_1},$ the elements $\Psi(s)$ and $\Psi(t)$ are not conjugate in $\Omega_\infty$. In particular, $\Psi$ is injective.
        \item The map $\phi_n\circ\Psi\colon\{0,1\}^{\N_1}\to \F_2$ is given by
        $$(s_i)_{i\ge 1}\mapsto 1+\sum_{i=1}^ns_i\bmod 2.$$
        Therefore the map
        $$(\phi_n)_{n\ge 1}\circ \Psi \colon \{0,1\}^{\N_1} \to \F_2^{\N_1}\cong \Omega_\infty^{\text{ab}}$$
        is a bijection.
        \item For every $s\in \{0,1\}^{\N_1}$, the element $\Psi(s)$ is settled.       
    \end{enumerate}
\end{proposition}
\begin{proof}
    $(a)$ This is done by induction on $n$. For $n=1$ the claim is obvious, because $\Psi(s)$ acts as a transposition on $V_1$ if $s_1=0$ and as the identity element if $s_1=1$. Suppose now that the claim is true for $n$, and let $c_1\ldots c_{d_n}$ be the decomposition in disjoint cycles of the action of $\Psi(s)$ on $V_n$, where the cycles are ordered by increasing length. By the inductive hypothesis, $d_n=1+\sum_{i=1}^ns_i$. Let $(v_{i_1},\ldots,v_{i_\ell})$ be the orbit corresponding to one of the $c_k$'s. If $k\ge 2$ or if $k=1$ and $s_{n+1}=0$ there exists a unique index $i_k\in \{i_1,\ldots,i_{\ell}\}$ such that $\varepsilon_{n+1}^{i_k}(\Psi(s))=1$. By Lemma \ref{orbits}, the descendants of $\{v_{i_1},\ldots,v_{i_\ell}\}$ form a unique $\Psi(s)$-orbit. Since the orbits corresponding to the $c_i$'s partition $V_n$, the orbits of their descendants partition $V_{n+1}$. It follows that if $s_{n+1}=0$ the number of orbits of $\Psi(s)$ acting on $V_n$ equals the number of orbits of $\Psi(s)$ acting on $V_{n+1}$, and the claim follows. On the other hand, if $k=1$ and $s_{n+1}=1$ consider the orbit $\mathcal O$ in $V_n$ corresponding to $c_1$. By construction of $\Psi(s)$, together with Lemma \ref{orbits}, the descendants of the nodes in $\mathcal O$ form two distinct orbits of length $m$. Hence the number of orbits of the action of $\Psi(s)$ on $V_{n+1}$ is $1$ plus the number of orbits of the action on $V_n$, and the claim follows again.
    
    $(b)$ This follows immediately from $(a)$: if two elements of $\Omega_\infty$ are conjugate, the number of orbits of their action on $V_n$ must coincide for every $n$.
    
    $(c)$ The construction of the portrait of $\Psi(s)$ shows that for every $n\ge 1$, the number of $\varepsilon_n^i(\Psi(s))$ that are equal to $1$ is $d_{n-1}$ if $s_n=0$ and $d_{n-1}-1$ if $s_n=1$, where $d_{n-1}$ is the number of orbits of $\Psi(s)$ acting on $V_{n-1}$. Hence by $(a)$ we have that:
    $$\phi_n(\Psi(s))=1+\sum_{i=1}^ns_i \bmod 2.$$

    $(d)$ The proof of $(a)$ shows that if $c_1\ldots c_m$ is the decomposition in disjoint cycles of the action of $\Psi(s)$ on $V_n$, where the cycles are ordered by increasing length, then if $s_{n+1}=0$ the decomposition in disjoint cycles of the action of $\Psi(s)$ on $V_{n+1}$ is $c_1'\ldots c_m'$ where the length of $c_i'$ is twice the length of $c_i$ for every $i$, while if $s_{n+1}=1$ then the decomposition in disjoint cycles of the action of $\Psi(s)$ on $V_{n+1}$ is $c_0' c_1'\ldots c_m'$, where for every $i\ge 2$ the length of $c_i'$ is twice the length of $c_i$ while $c_0'$ and $c_1'$ have the same length as $c_1$. This proves that the only nodes that potentially do not lie in a stable cycle are those corresponding to the orbit of $c_1$. Now, if $s_i$ is eventually equal to $0$, then from a certain point on the number of orbits of $\Psi(s)$ acting on $V_n$ is constant, and hence $\Psi(s)$ is strongly settled. If on the other hand $m(n)\coloneqq\sum_{i=1}^ns_i\to +\infty$, notice that the same argument as above shows that the length of $c_1$ is $2^{n-m(n)}$, and hence the proportion of stable cycles is at least
    $$\frac{2^n-2^{n-m(n)}}{2^n},$$
    that tends to $1$ when $n\to +\infty$ because $m(n)\to +\infty$.
    \end{proof}
    
    \section{Settled elements that cannot be Frobenius}
    In this section, we prove that uncountably many settled elements among the ones constructed in Section \ref{constructing_settled} can never be the Frobenius of a prime in an arboreal Galois group.
    \begin{lemma}\label{preperiodic_orbit}
    	Let $p$ be a prime, let $K/\Q_p$ be a finite extension with valuation ring $\O_K$ and maximal ideal $\m_K$. Let $f=x^2+c\in K[x]$, $\alpha\in K$ and set $c_{1,\alpha}=\alpha-c$ and $c_{n,\alpha}=f^n(0)-\alpha$ for $n\ge 2$. Suppose that $c_{n,\alpha}\ne 0$ for every $n$ and that $f(\alpha)\ne \alpha$. Then there exist $a\in \N_1$ and $b\in \N$ such that the sequence $(c_{an+b,\alpha})_ {n\ge 1}$ is preperiodic in $K^{\times}/{K^{\times}}^2$.
    \end{lemma}
	\begin{proof}
		Let $v$ be the valuation of $K$, normalized so that $v(\pi)=1$ for a uniformizer $\pi\in \O_K$. We recall that there exists some positive integer $M_K$ such that for every $N\ge M_K$ the map
        $$(\O_K/\pi^{N})^{\times}\to K^{\times}/{K^{\times}}^2$$
        that sends the residue class of $a\in \O_K^{\times}$ to $a$ modulo squares is well-defined (if $p$ is odd we can choose $M_K=1$ by Hensel's lemma). This means that the class modulo squares of a non-zero element $\pi^r\cdot u$ with $u\in \O_K^{\times}$ can be read off from $r$ and the reduction of $u$ modulo $\pi^{N}$.

         Let $S\subseteq \O_K$ be a system of representatives of $\O_K/\m_K$ containing $0$, so that $S$ is a finite set with the property that every $u\in K$ can be written as $\sum_{i\ge m}a_i\pi^i$, with $a_i\in S$ for every $i$, for some $m\in\Z$. 
		
		First assume that $v(c)=-e<0$, so that $v(f^n(0))=-2^{n-1}e$ for every $n\ge 1$. Choose $N\ge 1$ such that 
        $$2^{N-1}e>\max\{e+M_K,M_K-v(\alpha)\}.$$
        We have that $f^{N-1}(0)=\frac{1}{\pi^{2^{N-2}e}}(u+v\pi^{M_K})$ 
        with $u=a_0+a_1\pi+\ldots+a_{M_K-1}\pi^{M_K-1}$ for some $a_0,\ldots,a_{M_K-1}\in S$ with $a_0\ne 0$ and $v\in \O_K$. Since $2^{N-1}e>e+M_K$, it follows that
		$$f^N(0)=\frac{1}{\pi^{2^{N-1}e}}(u^2+v_0\pi^{M_K})$$
  for some $v_0\in \O_K$, and since $2^{N-1}e>M_K-v(\alpha)$ we also have that
        $$c_{N,\alpha}=\frac{1}{\pi^{2^{N-1}e}}(u^2+v_1\pi^{M_K})$$
        for some $v_1\in \O_K$. Now an easy induction shows that for every $n\ge N$ we have:
        $$c_{n,\alpha}=\frac{1}{\pi^{2^{n-1}e}}(u^{2^{n+1-N}}+v_{n+1-N}\pi^{M_K})$$
        for some $v_{n+1-N}\in \O_K$. Therefore the class of $c_{n,\alpha}$ in $K^{\times}/{K^{\times}}^2$ can be read off from the class of $u^{2^{n+1-N}}$ modulo $\pi^{M_K}$. But of course the sequence $(u^{2^{n+1-N}}\bmod \pi^{M_K})_{n\ge N}$ is preperiodic, and hence we get the claim.

        Next, we assume that $c\in \O_K$ and $v(\alpha)=-r<0$. Write
        $$c=u+v_0\pi^{M_K},$$
        with $u=a_0+a_1\pi+\ldots+a_{M_K-1}\pi^{M_K-1}$ for some $a_0,\ldots,a_{M_K-1}\in S$ and $v_0\in \O_K$. Let
        $$\alpha=\frac{1}{\pi^r}(s+t\pi^{M_K}),$$
        with $s=b_0+b_1\pi+\ldots+b_{M_K-1}\pi^{M_K-1}$ for some $b_0,\ldots,b_{M_K-1}\in S$ with $b_0\ne 0$ and $t\in \O_K$. Now by induction for every $n\ge 1$ we have:
        $$f^n(0)=g^n(0)+v_n\pi^{M_K}$$
        for some $v_n\in \O_K$, where $g=x^2+u\in \O_K[x]$. Consequently, for every $n\ge 2$ we have
        $$c_{n,\alpha}=\frac{1}{\pi^r}(-s+g^n(0)\pi^r+y_n\pi^{M_K}),$$
        for some $y_n\in \O_K$. Now clearly $-s+g^n(0)\pi^r\in \O_K^{\times}$ and of course the sequence $(s+g^n(0)\pi^r\bmod \pi^{M_K})_{n\ge 1}$ is preperiodic. Since the class of $c_{n,\alpha}$ modulo squares can be read off from this sequence, we get the claim.
		
        If $v(c)$ and $v(\alpha)$ are both non-negative, we are going to use the following fact: for infinitely many $N\in \N$, the sequence $(c_{n,\alpha}\bmod \pi^N)_{n\ge 1}$ is not eventually $0$. In fact, if for every $N$ large enough the sequence was eventually $0$, then $f^{n}(0)=\alpha=f^{n+1}(0)\bmod \pi^N$ for some $n$, and hence $f(\alpha)\equiv \alpha\bmod \pi^N$. But then $\pi^N\mid f(\alpha)-\alpha$ for every large enough $N$, implying that $f(\alpha)=\alpha$ and hence contradicting the hypothesis.

        Now let $\delta\coloneqq f(\alpha)-\alpha$, and let $r\in \N$ be the valuation of $\delta$. Now choose $N>M_K+r$ such that the periodic part of the sequence $(c_{n,\alpha}\bmod \pi^N)_{n\ge 1}$ contains a non-zero element. Such $N$ exists because of the above paragraph. Now since
        $$c_{n+1,\alpha}=f(\alpha+c_{n,\alpha})-\alpha=\delta+2\alpha c_{n,\alpha}+c_{n,\alpha}^2,$$
        we see that if $c_{n,\alpha}\bmod \pi^N$ is in the periodic part of the sequence and $v(c_{n,\alpha})>r$ then $v(c_{n+1,\alpha})=r$. It follows that the periodic part of $(c_{n,\alpha}\bmod \pi^N)_{n\ge 1}$ contains an element $c_{k,\alpha}$ of valuation $e\le r$. Write $c_{k,\alpha}=u_k\pi^e$ with $u_k\in \O_K^\times$. Then there exists some $\ell>k$ such that $c_{\ell,\alpha}\equiv c_{k,\alpha}\bmod \pi^{N}$, and if we write $c_{\ell,\alpha}=u_\ell\pi^e$ for some $u_\ell\in \O_K^\times$ we see that $u_k\equiv u_\ell\bmod \pi^{N-e}$. Since $N-e> M_K$, the class modulo squares of the subsequence $(c_{k+(\ell-k)n,\alpha})_{n\ge 1}$ is constant.
	\end{proof}
	\begin{corollary}\label{cor_per}
		Let $K$ be a finite extension of $\Q_p$, let $f=x^2+c\in K[x]$ and $\alpha\in K$ be such that $c_{n,\alpha}\ne 0$ for every $n\ge 1$ and $f(\alpha)\ne \alpha$. If $\sigma\in G_\infty(f,\alpha)\subseteq \Omega_\infty$, then there exist $a\in \N_1$, $b\in \N$ such that the sequence $(\phi_{an+b}(\sigma))_{n\ge 1}\in \F_2^{\N_1}$ is preperiodic.
	\end{corollary}
	\begin{proof}
	By Lemma \ref{preperiodic_orbit}, there exist $a\in \N_1$ and $b\in \N$ such that the sequence $(c_{an+b,\alpha})_{n\ge 1}$ is preperiodic in $K^{\times}/{K^{\times}}^2$. Now by \cite[Proposition 4.2]{ferpag1} (see \cite[Lemma 2.9]{ferpag2} for a more general version), for every $i\ge 1$ the value $\phi_i(\sigma)$ coincides with $\chi_{c_{i,\alpha}}(\sigma)$, where for a fixed square root $r$ of $c_{i,\alpha}$ in $\overline{K}$ we have that $\sigma(r)=(-1)^{\chi_{c_{i,\alpha}}(\sigma)}r$. But of course if $(c_{an+b,\alpha})_{n\ge 1}$ is preperiodic in $K^{\times}/{K^{\times}}^2$ then so is $(\chi_{c_{an+b,\alpha}}(\sigma))_{n\ge 1}$, proving the claim.
		\end{proof}

  We are now ready to state and prove our main result. Recall that we denote by $\Psi\colon \{0,1\}^{\N_1}\to \Omega_\infty$ the map constructed in Section \ref{constructing_settled}.

  \begin{theorem}\label{non_frobenius}
      Let $s=(s_i)_{i\ge 1}\in \{0,1\}^{\N_1}$ be such that $s_1=0$ and that $(\phi_{an+b}(\Psi(s)))_{n\ge 1}$ is not preperiodic for every $a\in \N_1$ and $b\in \N$. Let $\sigma\coloneqq\Psi(s)\in \Omega_\infty$. Let $K$ be a number field and let $f=x^2+c\in K[x]$ and $\alpha\in K$ be such that $c_{n,\alpha}\ne 0$ for every $n\ge 1$. Let $G_\infty(f,\alpha)\subseteq \Omega_\infty$ be the image of the arboreal representation associated to $(f,\alpha)$, and suppose that $\sigma\in G_\infty(f,\alpha)$. Then $\sigma$ is a settled element that is not a Frobenius for any prime $\p$ of $K$.
  \end{theorem}
  \begin{proof}
    That $\sigma$ is settled follows from Proposition \ref{settled_properties}. Since $s_1=0$, $\sigma$ acts non-trivially on the first level of the tree. Since by hypothesis $\sigma \in G_\infty(f,\alpha)$, it follows that $f-\alpha$ is irreducible. In particular, $f(\alpha)\ne \alpha$. Suppose that $\sigma$ was a Frobenius at a prime $\p$ of $K$. This means that $\sigma$ belongs to the decomposition group $G_\p\subseteq G_\infty(f,\alpha)\subseteq \Omega_\infty$; it follows that $\sigma$ belongs to the image of the arboreal representation of $(f,\alpha)$ over the completion of $K$ at $\p$. By Corollary \ref{cor_per} there is therefore an arithmetic progression $(an+b)_{n\ge 1}$ such that $(\phi_{an+b}(\sigma))_{n\ge 1}$ is preperiodic, contradicting the hypothesis.
  \end{proof}

  \begin{remark}\label{explicit}
      The set of elements of $\{0,1\}^{\N_1}$ that satisfy the hypotheses of Theorem \ref{non_frobenius} is uncountable. This can be viewed with a measure theoretic argument, but it is also possible to construct an explicit uncountable set $S$ of sequences that satisfy the hypotheses of the theorem in the following way. Let $\widetilde{s}$ be the sequence 
      $$(1,0,1,1,0,0,1,1,1,0,0,0,1,1,1,1,0,0,0,0,\ldots)\in \F_2^{\N_1}.$$
      This cannot contain any subsequence of the form $(\widetilde{s}_{an+b})_{n\ge 1}$ with $a\in \N_1$ and $b\in \N$ that is preperiodic; indeed if this was the case then it would also contain a subsequence of this form that is eventually constant. However it is easy to see that any subsequence of $\widetilde{s}$ the form $(\widetilde{s}_{an+b})_{n\ge 1}$ necessarily contains infinitely many $0$'s and $1$'s, because $\widetilde{s}$ contains infinitely many arbitrarily long blocks of $0$'s and $1$'s. Hence the element $s\in \{0,1\}^{\N_1}$ such that $((\phi_n)_{n\ge 1}\circ\Psi)(s)=\widetilde{s}$ satisfies the hypotheses of Theorem \ref{non_frobenius}. To construct uncountably many such elements,  let $(n_k)_{k\ge 1}$ be the sequence of indices of the first digit of each block of $0$ and $1$ in $\widetilde{s}$ starting with the first block of $0$'s, so that $(n_k)=(2,3,5,7,10,13,17,\ldots)$. For any sequence $t=(t_k)_{k\ge 1}\in \{0,1\}^{\N_1}$, consider the sequence $(\widetilde{s}(t)_n)_{n\ge 1}$ defined as:
      $$\widetilde{s}(t)_n=\begin{cases}
          \widetilde{s}_n & \mbox{ if } n\notin \{n_k\}_{k\ge 1}\\
          t_{k} & \mbox{ if } n=n_k \mbox{ for some } k \ge 1\\
      \end{cases}.$$
      Then $(\widetilde{s}(t)_n)_n$ clearly satisfies again the hypotheses of Theorem \ref{non_frobenius}, for the same reason for which $s$ does, and of course the set $\widetilde{S}\coloneqq\{\widetilde{s}(t)\colon t\in \{0,1\}^{\N_1}\}$ is uncountable. Now by Proposition \ref{settled_properties}, the set $S=((\phi_n)_{n\ge 1}\circ\Psi)^{-1}(\widetilde{S})\subseteq \{0,1\}^{\N_1}$ is uncountable. Using Proposition \ref{settled_properties}, one can describe explicitly the elements of $S$. For example, if
      $$s=(0,1,1,0,1,0,1,0,0,1,0,0,1,0,0,0,1,0,0,0,\ldots,)$$
      then $(\phi_n)_{n\ge 1}(\Psi(s))=\widetilde{s}.$
  \end{remark}

  \bibliographystyle{plain}
\bibliography{bibliography}

\begin{thebibliography}{1}

\bibitem{bosjon2}
Nigel Boston and Rafe Jones.
\newblock The image of an arboreal {G}alois representation.
\newblock {\em Pure Appl. Math. Q.}, 5(1):213--225, 2009.

\bibitem{lukina}
Mar\'{\i}a~Isabel Cortez and Olga Lukina.
\newblock Settled elements in profinite groups.
\newblock {\em Adv. Math.}, 404(part B):Paper No. 108424, 56, 2022.

\bibitem{ejder}
\"{O}zlem Ejder.
\newblock Galois theory of quadratic rational functions with periodic critical
  points.
\newblock {\em J. Number Theory}, 280:212--245, 2026.

\bibitem{ferpag3}
Andrea Ferraguti and Carlo Pagano.
\newblock Constraining images of quadratic arboreal representations.
\newblock {\em Int. Math. Res. Not. IMRN}, (22):8486--8510, 2020.

\bibitem{ferpag2}
Andrea Ferraguti and Carlo Pagano.
\newblock Abelian dynamical {G}alois groups for unicritical polynomials.
\newblock https://arxiv.org/abs/2303.04783, 2023.

\bibitem{ferpag1}
Andrea Ferraguti, Carlo Pagano, and Daniele Casazza.
\newblock The inverse problem for arboreal {G}alois representations of index
  two.
\newblock https://arxiv.org/abs/1907.08608, 2019.

\bibitem{jones}
Rafe Jones.
\newblock Galois representations from pre-image trees: an arboreal survey.
\newblock In {\em Actes de la {C}onf\'{e}rence ``{T}h\'{e}orie des {N}ombres et
  {A}pplications''}, volume 2013 of {\em Publ. Math. Besan\c{c}on Alg\`ebre
  Th\'{e}orie Nr.}, pages 107--136. Presses Univ. Franche-Comt\'{e},
  Besan\c{c}on, 2013.

\bibitem{bosjon1}
Rafe Jones and Nigel Boston.
\newblock Settled polynomials over finite fields.
\newblock {\em Proc. Amer. Math. Soc.}, 140(6):1849--1863, 2012.

\end{thebibliography}
\end{document}